\documentclass[10pt,reqno,a4paper]{amsart}

\usepackage{amssymb,amsrefs,amscd,amsmath, fullpage}
\usepackage{graphicx,nicefrac}
\usepackage[shortlabels]{enumitem}
\usepackage{mathrsfs}
\usepackage{hyperref}
\usepackage{ourbib}
\usepackage[all]{xy}
\usepackage{callouts}

\setlist{leftmargin=8mm}

\hypersetup{
    colorlinks,
    linkcolor={red!50!black},
    citecolor={blue!50!black},
    urlcolor={blue!80!black}
}

\newcommand{\K}{\mathscr{K}}

\newcommand{\R}{\mathbb{R}}
\newcommand{\N}{\mathbb{N}}
\newcommand{\dM}{{\partial M}}

\renewcommand{\phi}{\varphi}

\newcommand{\ch}{\mathsf{ch}}
\renewcommand{\c}{\mathsf{c}}
\newcommand{\scal}{\mathrm{scal}}

\newcommand{\<}{\left\langle}
\renewcommand{\>}{\right\rangle}

\newcommand{\cw}{\text{-}\mathrm{cw}_2}
\newcommand{\rk}{\mathrm{rk}}

\DeclareMathOperator{\id}{\mathrm{id}}

\newtheorem{theorem}{Theorem}
\newtheorem{lemma}[theorem]{Lemma}

\newtheorem{corollary}[theorem]{Corollary}

\theoremstyle{definition}
\newtheorem{remark}[theorem]{Remark}
\newtheorem{definition}[theorem]{Definition} 
\newtheorem{example}[theorem]{Example}

\begin{document}

\title{\emph{K}-cowaist of manifolds with boundary} 


\author{Christian B\"ar}
\address{Universit\"at Potsdam, Institut f\"ur Mathematik, 14476 Potsdam, Germany}
\email{\href{mailto:cbaer@uni-potsdam.de}{cbaer@uni-potsdam.de}}
\urladdr{\url{https://www.math.uni-potsdam.de/baer}}
\author{Bernhard Hanke}
\address{Universit\"at Augsburg, Institut f\"ur Mathematik, 86135 Augsburg, Germany}
\email{\href{mailto:hanke@math.uni-augsburg.de}{hanke@math.uni-augsburg.de}}
\urladdr{\url{https://www.math.uni-augsburg.de/diff/hanke}}

\begin{abstract} 
We extend the $K$-cowaist inequality to generalized Dirac operators in the sense of Gromov and Lawson and study applications to manifolds with boundary. 
\end{abstract}





\thanks{\emph{Acknowledgment.} 
We thank the Special Priority Programme SPP 2026 ``Geometry at Infinity'' funded by Deutsche Forschungsgemeinschaft for financial support. 
B.H.\ thanks the University of Potsdam for its hospitality.}

\date{September 14, 2023}

\maketitle

\section{Introduction} 

The notion of $K$-cowaist of closed oriented smooth Riemannian manifolds was introduced by Gromov in \cite{G1996}*{Section 4} under the name $K$-area.
It is defined as the inverse of the infimum of the operator norms of the curvatures of all smooth Hermitian vector bundles over the given manifold with at least one non-zero Chern number. 
Its main application is to the scalar curvature geometry of spin manifolds, see the ``$K$-area inequality'' in  \cite{G1996}*{Section 5$\tfrac{1}{4}$} . 

The notion was later generalized to homology classes in smooth Riemannian manifolds by Listing in \cite{Listing2013}, and in simplicial complexes by Hunger in \cite{Hunger2019}, to infinite dimensional bundles by Hanke and Hunger in \cites{Hanke2012, Hunger2021} and to manifolds with boundary by Bär, Hanke, and Listing in \cites{Listing2013, BH2023}. 
The behavior of $K$-cowaist under surgery has been studied by Fukumoto in \cite{Fukumoto2015}.

An important step in this discussion is the construction of well-behaved twist bundles for the classical Dirac operator on spin manifolds from vector bundles with non-zero Chern numbers. 
The elegant argument, which uses $K$-theoretic Adams operations together with a ``trivial algebraic lemma'' on formal power series, was sketched by Gromov in \cite{G1996}*{Section 5$\tfrac{3}{8}$}, and more detailed expositions were later given by Bär, Hanke, and Listing in \cites{Listing2013, BH2023}. 
Recently, another exposition was given by Wang in \cite{Wang2023}.

In the present note, which is in line with \cite{BH2023}, we extend the discussion of $K$-cowaist in two directions.
Firstly, rather than confining ourselves to the spinorial Dirac operator, we allow arbitrary generalized Dirac operators in the sense of Gromov and Lawson and associate to each such operator a notion of $\omega$-cowaist.
Here $\omega$ represents the index form of the operator occurring in the Atiyah-Singer index theorem.
In the case of the spinorial Dirac operator, $\omega$ is the $\hat{A}$-form.
Theorem~\ref{thm:WaistComparison} compares the $K$-cowaist and the $\omega$-cowaist. 
Interestingly, the constant occurring in this $K$-cowaist inequality depends only on the dimension of the manifold but not on the choice of operator.

Secondly, in the application we allow the manifolds to have boundary.
Using the spinorial Dirac operator, one recovers the known fact that a compact spin manifold with infinite $K$-cowaist does not support a Riemannian metric of positive scalar curvature such that the boundary becomes mean convex.
For compact spin$^c$ manifolds with infinite $K$-cowaist we find that there is no Riemannian metric such that the scalar curvature dominates any $2$-form representing the Chern class of the determinant bundle and, again, such that the boundary becomes mean convex.
Remarkably, the condition on the boundary is always mean convexity, irrespective of the choice of operator.

\section{\texorpdfstring{$K$-cowaist}{K-cowaist}} 
\label{infiniteKarea} 

Let $M$ be an oriented compact smooth Riemannian manifold with or without boundary.
We call a Hermitian vector bundle $E$ over $M$ with connection \emph{boundary-adapted} if it is isomorphic to the trivial bundle with trivial connection over a neighborhood of the boundary.
We call $E$ \emph{admissible} if it is boundary-adapted and it has at least one nontrivial Chern number.
The latter means that there are $\gamma_j\in\N_0$ such that 
$$
\int_M \c_{\gamma_1}(E) \wedge \cdots \wedge \c_{\gamma_m}(E) \neq 0.
$$
Here $\c(E) = \c_0(E)+\c_1(E)+\ldots+\c_m(E) = 1+\c_1(E)+\ldots+\c_m(E)$ is the Chern form of $E$.
Admissible bundles can exist only on even-dimensional manifolds because $\c_j(E)$ has even degree $2j$.
Indeed, the dimension of $M$ satisfies $n=2(\gamma_1+\ldots+\gamma_m)$.

Equivalently, one may demand that 
$$
\int_M \ch_{\gamma_1}(E) \wedge \cdots \wedge \ch_{\gamma_m}(E) \neq 0
$$
for some $\gamma_j\in\N_0$.
Here $\ch(E) = \ch_0(E)+\ch_1(E)+\ldots+\ch_m(E) = \mathrm{rank}(E)+\ch_1(E)+\ldots+\ch_m(E)$ is the Chern character form of $E$.
The Chern numbers and the Chern character numbers can be expressed as linear combinations of each other.

Note that the support of the curvature $R^E$ and hence that of $\c_j(E)$ and $\ch_j(E)$ for $j\ge1$ is contained in the interior of $M$ because $M$ is boundary-adapted.

Given a Hermitian vector bundle with connection over a Riemannian manifold $M$, let $R^E$ be its curvature tensor.
We define its norm by
\[
\|R^E\| := \sup_{x\in M} \sup_{\genfrac{}{}{0pt}{}{X,Y\in T_xM}{|X|=|Y|=1}} |R^E(X,Y)|
\]
where $|R^E(X,Y)|$ is the operator norm of the endomorphism $R^E(X,Y)$.
The rank of $E$ is denoted by $\rk(E)$.

\begin{definition}\label{def:Kcowaist}
The \emph{$K$-cowaist} of an oriented compact Riemannian manifold $M$ with (possibly empty) boundary is defined by
\[
K\cw(M) :=    \frac{1}{\inf\{\|R^E\|\mid E \text{ is an admissible bundle over }M\}} \in [0,\infty] .
\]
\end{definition}

\begin{remark}
If we replace the Riemannian metric $g$ on $M$ by $\lambda^2g$, where $\lambda$ is a positive constant, the operator norm $|R^E(X,Y)|$ for $X, Y \in T_x M$ remains unchanged, whereas $\|R^E\|$ is replaced by $\lambda^{-2}\|R^E\|$.
Therefore, $K\cw(M)$ is replaced by $\lambda^2K\cw(M)$, thus $K\cw(M)$ scales like an area.
This motivates the terminology $K$-area for $K\cw(M)$ as introduced by Gromov in \cite{G1996}.
In \cite{Gromov2020}, Gromov argues that the term $K$-cowaist is more appropriate.
\end{remark}

\begin{remark}
The condition $K\cw(M)=\infty$ is independent of the metric on $M$ since $M$ is compact and any two metrics can be bounded by each other.
There is a rich class of manifolds satisfying this condition, including enlargeable manifolds, see \cite{BH2023}, for example. 
If $M$ is connected and without boundary, the condition $K\cw(M) = \infty$ only depends on the image of the fundamental class of $M$ in the rational homology of $B\pi_1(M)$ under the classifying map of the universal cover of $M$, see Corollary~7.4 in \cite{Hunger2019}.  
\end{remark}

\begin{definition}
Let $\omega=1+\omega_1 + \ldots + \omega_m$ be a smooth mixed differential form on the oriented compact Riemannian manifold $M$ with boundary, where $\omega_j$ has degree $2j$.
The $\omega$-cowaist of $M$ is defined by 
\begin{align*}
\omega\cw(M) &:=  \frac{1}{\inf\big\{\|R^E\|\mid E \text{ is boundary-adapted and } \int_M\omega\wedge [\ch(E)-\rk(E)] \neq 0\big\} } 
\in [0,\infty] .
\end{align*}
\end{definition}

The following lemma is the key to comparing the $K$-cowaist and the $\omega$-cowaist.
The idea goes back to Gromov \cite{G1996} and the lemma is essentially already contained as Lemma~7 in \cite{BH2023}.
For the reader's convenience, we provide the full (short) proof here.

\begin{lemma}\label{lem:KversusO}
Let $M$ be an oriented compact Riemannian manifold of even dimension $n=2m$ with boundary.
Let $E$ be an admissible bundle.
Let $\omega=1+\omega_1 + \ldots + \omega_m$ be a smooth mixed differential form on $M$ where $\omega_j$ has degree $2j$.
   
Then there exists a boundary-adapted bundle $E'$ over $M$ such that 
\begin{equation}
\int_M \omega\wedge[\ch(E')-\rk(E')] \neq 0
\label{eq:IntNotTriv}
\end{equation}
and 
\begin{equation}
\|R^{E'}\|\le c(m) \|R^E\|
\label{eq:CurvEst}
\end{equation}
where $c(m)$ is a constant only depending on $m$.
\end{lemma}

\begin{proof}
For $k\in\N_0$ there is a virtual bundle $\Psi_kE = \Psi_k^+E - \Psi_k^-E$ with the property 
\begin{equation}
\ch_j(\Psi_k E)=\ch_j(\Psi_k^+ E)-\ch_j(\Psi_k^- E)=k^j\ch_j(E).
\label{eq:chAdams}
\end{equation}
Here $\Psi_k$ is known as the $k$th Adams operation.
The case $j=0$ shows that the Adams operations $\Psi_k$ preserve the rank.
Both bundles $\Psi_k^+E$ and $\Psi_k^- E$ are universal expressions in tensor products of exterior products of $E$, see \cite{A1989}*{Section~3.2} for details.

For a multi-index $k=(k_1,\ldots,k_m)$ we put 
$$
\Psi_kE := \Psi_{k_1}E\otimes\cdots\otimes\Psi_{k_m}E
$$
and rewrite this virtual bundle as a difference of honest bundles by
\begin{align*}
\Psi_kE 
&=
\bigoplus_{\genfrac{}{}{0pt}{}{\mathrm{even\,\#}}{\mathrm{of}\,-\,\mathrm{'s}}}\Psi_{k_1}^\pm E\otimes\cdots\otimes\Psi_{k_m}^\pm E -\bigoplus_{\genfrac{}{}{0pt}{}{\mathrm{odd\,\#}}{\mathrm{of}\,-\,\mathrm{'s}}}\Psi_{k_1}^\pm E\otimes\cdots\otimes\Psi_{k_m}^\pm E 
=: \Psi^+_kE - \Psi^-_kE .
\end{align*}
Again, $\Psi_k^+E$ and $\Psi_k^- E$ are universal expressions in tensor products of exterior products of $E$.
Hence, they inherit natural Hermitian metrics and connections, and they are boundary-adapted.
In particular,
\begin{equation}
\|R^{\Psi_k^\pm E}\|\le c_k \|R^E\|
\label{eq:AdamsR}
\end{equation}
where the constant $c_k$ depends only on $k$.
Note that $\rk(\Psi_k^+E)-\rk(\Psi_k^-E)=\rk(\Psi_kE)=\rk(E)^m$. 

For $k=(k_1,\ldots,k_m)\in\N_0^m$ we put 
\begin{align*}
P(k_1,\ldots,k_m)
&:=
\int_M \omega\wedge[\ch(\Psi_{k}E)-\rk(\Psi_{k}E)] \\
&=
\int_M \omega\wedge[\ch(\Psi_{k_1}E)\wedge\cdots\wedge\ch(\Psi_{k_m}E)-\rk(E)^m] .
\end{align*}
Expanding $\omega = 1+\omega_1 + \ldots + \omega_{m}$ and the Chern characters yields, using \eqref{eq:chAdams},
$$
P(k_1,\ldots,k_m)
=
\sum_{\gamma_1+\ldots+\gamma_m=m} k_1^{\gamma_1}\cdots k_m^{\gamma_m} \int_M \ch_{\gamma_1}(E)\wedge\cdots\wedge\ch_{\gamma_m}(E) + \mbox{l.o.t.}
$$
where l.o.t.\ stands for terms of lower total order in $k_1,\ldots,k_m$.
In particular, $P$ is a polynomial in $k_1,\ldots,k_m$ of total degree at most $m$.

If $P(k_1,\ldots,k_m)=0$ held for all $k = (k_1, \ldots, k_m)  \in \{0,1,\ldots,m\}^m$, then $P$ would vanish as a polynomial, hence
$$
\int_M \ch_{\gamma_1}(E)\wedge\cdots\wedge\ch_{\gamma_m}(E) = 0
$$
for all $\gamma_i\in\N_0$ with $\gamma_1+\ldots+\gamma_m=m$, contradicting the admissibility of $E$.
Thus we can choose some $k\in \{0,1,\ldots,m\}^m$ such that $P(k)\neq0$, i.e.\ 
\begin{align*}
0 
&\neq 
\int_M \omega\wedge[\ch(\Psi_{k}E)-\rk(\Psi_{k}E)] \\
&=  
\int_M \omega\wedge[\ch(\Psi_{k}^+E)-\rk(\Psi_{k}^+E)]-\int_M \omega\wedge[\ch(\Psi_{k}^-E)-\rk(\Psi_{k}^-E)] .     
\end{align*}
Hence, $E'=\Psi_{k}^+E$ or $E'=\Psi_{k}^-E$ satisfies \eqref{eq:IntNotTriv}.
Equation~\eqref{eq:AdamsR} implies $\|R^{\Psi^\pm_kE}\|\le c(m) \|R^E\|$ since there are only finitely many possibilities for~$k$.
\end{proof}

\begin{theorem}\label{thm:WaistComparison}
Let $M$ be an oriented compact Riemannian manifold with boundary of even dimension $2m$. 
Let $\omega=1+\omega_1 + \ldots + \omega_m$ be a smooth mixed differential form on $M$ where $\omega_j$ has degree $2j$.
Then
\[
K\cw(M) \le c(m)\cdot \omega\cw(M)   
\]
where $c(m)$ is a constant which depends only on $m$.
\end{theorem}

\begin{proof}
If there are no admissible bundles over $M$, then $K\cw(M) = 0$ and there is nothing to show.
Thus, let $E\to M$ be admissible and let $E'$ be the corresponding bundle from Lemma~\ref{lem:KversusO}.
Then $\int_M \omega\wedge[\ch(E')-\rk(E)]\neq 0$ and
\[
c(m)^{-1}\cdot  \omega\cw(M)^{-1} \le  c(m)^{-1} \cdot\|R^{E'}\| \le \|R^E\| .
\]
Taking the infimum over all admissible $E$ concludes the proof.
\end{proof}

Note that the constant $c(m)$ does not depend on the form $\omega$.

\section{An application}

Let $M$ be a Riemannian manifold, let $S^+,S^-\to M$ be complex vector bundles equipped with Hermitian metrics and let $D$ be a differential operator of first order mapping sections of $S^+$ to sections of $S^-$.
We restrict our attention to operators such that $\begin{pmatrix}0 & D^* \\ D & 0\end{pmatrix} \colon S^+ \oplus S^- \to S^+ \oplus S^-$ is a generalized Dirac operator in the sense of Gromov and Lawson, see \cite{GL}*{Section~1}.
Here $D^*$ denotes the formally adjoint operator of $D$.
We then call $D$ a GL-Dirac operator for short.
After interchanging the roles of $S^+$ and $S^-$, the formal adjoint $D^*$ is again a GL-Dirac operator.

In particular, the symbol of $D$ defines a multiplication $TM \otimes S^{\pm} \to S^{\mp}$ satisfying the Clifford relations, and the bundles $S^\pm$ are equipped with metric connections $\nabla^{S^\pm}$ whose curvature tensors we denote by $R^{S^\pm}$.
The operator $D$ satisfies the Weitzenböck formulas
\begin{align*}
D^*D &= (\nabla^{S^+})^*\nabla^{S^+} + \K^+, \\
DD^* &= (\nabla^{S^-})^*\nabla^{S^-} + \K^-
\end{align*}
where $\K^\pm = \frac12\sum_{jk}e_j\cdot e_k\cdot R^{S^\pm}(e_j,e_k)$, see Proposition~2.5 in \cite{GL}.
Here  $(\nabla^{S^\pm})^*$ denotes the formally adjoint operator of $\nabla^{S^{\pm}}$.

Given a GL-Dirac operator $D$ and a Hermitian vector bundle with metric connection, one defines the \emph{twisted} Dirac operator $D^E$ locally by
\[
D^E = \sum_{j} (e_j\cdot \otimes \id)\nabla^{S^+\otimes E}_{e_j}
\]
for some local orthonormal frame $(e_1, \ldots, e_n)$ of $TM$. 
The twisted Dirac operator maps sections of $S^+\otimes E$ to sections of $S^-\otimes E$ and is again a GL-Dirac operator.

If $M$ is a Riemannian manifold of dimension $n$ with boundary $\dM$, then we denote by $H\in C^\infty(\dM,\R)$ the mean curvature of the boundary, defined as $\tfrac{1}{n-1}$ times the trace of the second fundamental form of $\dM \subset M$. 
The sign convention is such that $H$ is positive if the mean curvature vector field is inward pointing.
The mean curvature of a Euclidean ball is positive, for example.
We say that the boundary is \emph{mean convex} if $H\ge0$. 

Given the GL-Dirac operator $D$ on $M$, there is an adapted Dirac operator $A$ over $\dM$, acting on sections of $S^+|_{\dM}$.
It is defined by 
\[
A = -\nu\cdot D - \nabla^{S^+}_\nu + \tfrac{n-1}{2}H .
\]
Here $\nu$ is the inward pointing unit normal vector field along $\dM$.
The operator $A$ anticommutes with Clifford multiplication by $\nu$.
Performing the integration by parts in the Weitzenböck formula, we find for smooth sections $\phi$ of $S^+$:
\begin{equation}
\int_M \big[|D\phi|^2 - |\nabla^{S^+}\phi|^2 - \<\K^+\phi,\phi\>\big] = \int_{\dM} \<(\tfrac{n-1}{2}H-A)\phi,\phi\>,
\label{eq:WeitzenIntegrated}
\end{equation}
see Equation (27) in \cite{BB2}.
We say that a sufficiently smooth section $\phi$ of $S^+$ satisfies the \emph{strong Atiyah-Patodi-Singer (APS) boundary condition} if $\phi|_{\dM}$ is contained in the sum of the $L^2$-eigenspaces of $A$ to negative eigenvalues.
We say it satisfies the \emph{weak APS boundary condition} if $\phi|_{\dM}$ is contained in the sum of the eigenspaces to nonpositive eigenvalues.

Associated to $D$ there is a mixed differential form $\omega=1+\omega_1+\cdots+\omega_{m}$ where $\omega_j$ has degree $2j$ which is manufactured out of the short-time asymptotics of the corresponding heat kernel. 
By the Atiyah-Patodi-Singer index theorem \cite{APS1}, it has the property that each twisted operator $D^E$ has the index
\begin{align}
\mathrm{ind}(D^{E}) = \int_M\omega\wedge\ch(E) + \text{ boundary contribution},
\label{eq:APS0}
\end{align}
if we impose weak or strong APS boundary conditions.
Denoting by $A^E$ the adapted operator for the GL-Dirac operator $D^E$, the boundary contribution involves the $\eta$-invariant of $A^E$, a transgression term and the dimension of the kernel of $A^E$.

\begin{theorem} \label{thm:omega-waist-inequality} 
Let $M$ be a $2m$-dimensional compact oriented Riemannian manifold with (possibly empty) mean convex boundary $\dM$.
Let $D$ be a GL-Dirac operator with index form $\omega$.
Let $\K^\pm$ be the curvature terms in the Weitzenböck formulas for $D$.
Suppose that $\K^+\ge \kappa$ and $\K^-\ge \kappa$ in the sense of symmetric endomorphism where $\kappa>0$ is a positive constant.
Then
\[
\omega\cw(M) \le \frac{m(2m-1)}{\kappa} .
\]
\end{theorem}

\begin{proof}
Let $E\to M$ be boundary-adapted such that $\int_M \omega\wedge[\ch(E)-\rk(E)]\ne0$.
If there are no such bundles, then $\omega\cw(M)=0$ and there is nothing to show.
We write $r=\rk(E)$ and denote by $E_0^r$ the trivial flat bundle of rank $r$.
We impose the weak APS boundary condition.
Now \eqref{eq:APS0} yields
\begin{align}
\mathrm{ind}(D^{E}) &= \int_M\omega\wedge\ch(E) + \text{ boundary contribution},
\label{eq:APS1} \\
\mathrm{ind}(D^{E_0^r}) &= r\int_M\omega + \text{ boundary contribution}.
\label{eq:APS2}
\end{align}
The boundary contributions in \eqref{eq:APS1} and \eqref{eq:APS2} coincide because $E$ is boundary-adapted and hence $D^{E}$ and $D^{E_0^r}$ coincide in a neighborhood of $\dM$.
Therefore,
\[
\mathrm{ind}(D^{E}) - \mathrm{ind}(D^{E_0^r}) 
= 
\int_M \omega\wedge\ch(E) - r\int_M\omega 
= 
\int_M \omega\wedge[\ch(E) - \rk(E)] 
\neq 0 .
\]
It follows that $\mathrm{ind}(D^{E})\neq0$ or $\mathrm{ind}(D^{E_0^r})\neq0$.
We discuss the case $\mathrm{ind}(D^{E})\neq0$, the second case being even simpler.

If $\mathrm{ind}(D^{E})>0$, then we find a smooth section $\phi$ of $S^+\otimes E$ in the kernel of $D^{E}$.
Inserting this $\phi$ into \eqref{eq:WeitzenIntegrated} with $D^{E}$ instead of $D$, we get
\begin{align}
0
&=
\int_M \big[|\nabla^{S^+\otimes E}\phi|^2 + \langle\K^{E,+}\phi,\phi\rangle\big] + \int_{\dM} \<(\tfrac{2m-1}{2}H-A)\phi,\phi\> 
\ge
\int_M \langle\K^{E,+}\phi,\phi\rangle
\label{eq:pos}
\end{align}
since all other terms are nonnegative.
Here we use $H\ge0$ and the fact that $\phi$ satisfies the weak APS boundary condition.
 
The curvature term in the Weitzenböck formula for $D^{E}$ is given by 
\[
\K^{E,\pm}=\K^{\pm}\otimes\id+\tfrac12\sum_{j,k=1}^{2m} e_je_k\otimes R^{E}(e_j,e_k).
\]
The operator norm of the correction term satisfies
\[
\Big|\tfrac12\sum_{jk}e_je_k\otimes R^{E}(e_j,e_k)\Big| \le m(2m-1)\|R^{E}\| .
\]
If we had $m(2m-1)\|R^{E}\|<\kappa$, then $\K^{E,+}$ would be positive as an endomorphism, contradicting \eqref{eq:pos}.
Therefore, $\|R^{E}\|\ge\frac{\kappa}{m(2m-1)}$.

If $\mathrm{ind}(D^{E})<0$, the adjoint operator has a nontrivial kernel.
Since $D^E$ is formally selfadjoint and the strong APS boundary condition is adjoint to the weak APS boundary condition\footnote{This uses that $A$ anticommutes with $\nu$, see Theorem~4.6 and Example~5.12 in \cite{BB2}.}, this means that we find a smooth section $\phi$ of $S^-\otimes E$ satisfying the strong APS boundary condition with $(D^{E})^*\phi=0$.
Now the proof proceeds as before and we again get $\|R^E\| \ge \frac{\kappa}{m(2m-1)}$.

Taking the infimum over all $E$ yields $\omega\cw(M)\le \frac{m(2m-1)}{\kappa}$.
\end{proof}

Theorems~\ref{thm:WaistComparison} and \ref{thm:omega-waist-inequality} combine to give

\begin{corollary}
Under the assumptions of Theorem~\ref{thm:omega-waist-inequality} we have 
\[
K\cw(M) \le \frac{c_1(m)}{\kappa} 
\]
where $c_1(m)$ is a positive constant which depends only on $m$.
\hfill\qed
\end{corollary}

\begin{corollary}\label{cor:nonexist}
An oriented even-dimensional compact differentiable manifold $M$ with $K\cw(M)=\infty$ does not admit a Riemannian metric with mean convex boundary and a GL-Dirac operator such that $\K^+$ and $\K^-$ are both positive.
\hfill\qed
\end{corollary}

\begin{example}
If $M$ is a spin manifold and $D$ is the spinorial Dirac operator, then $\K^\pm=\frac{\scal}{4}$ where $\scal$ denotes the scalar curvature of $M$.
Corollary~\ref{cor:nonexist} says that compact spin manifolds $M$ with $K\cw(M)=\infty$ do not admit metrics with positive scalar curvature and mean convex boundary, cf.\ Theorem~19 in \cite{BH2023}.
In this example, $\omega=\hat{A}$ is the $\hat{A}$-form.
\end{example}

\begin{example}
Let $M$ be a compact Riemannian spin$^c$ manifold with determinant bundle $L$.
Let $\Omega$ be a real $2$-form such that $\frac{\Omega}{2\pi}$ represents the first Chern class of $L$ in deRham cohomology.
Then there exists a metric connection on $L$ whose curvature is given by $i\Omega$.
The curvature contribution to the Weitzenböck formula for the spinorial Dirac operator is given by $\K^\pm=\frac{\scal}{4}- \frac{1}{2}\Omega$ where $\Omega$ acts by Clifford multiplication, see Theorem~D.12 in \cite{LM}.

At each point of $M$, we can find an orthonormal basis $e_1,\dots,e_m,e_{m+1},\dots,e_{2m}$ of the cotangent space such that 
\[
\Omega = \sum_{j=1}^m \lambda_j e_j\wedge e_{m+j}
\]
for $\lambda_j\in\R$.
The norm $|\Omega|=\sum_{j=1}^m|\lambda_j|$ is defined independently of the choice of basis.
For any spinor $\phi$ we have 
\[
|\<\Omega\cdot\phi,\phi\>| \le  \sum_{j=1}^m|\<\lambda_je_j\cdot e_{m+j}\cdot\phi,\phi\>| \le |\Omega||\phi|^2.
\]
Thus, $\K^\pm$ is positive provided $\scal > 2|\Omega|$.
Hence, if $\Omega$ is a real $2$-form such that $\frac{\Omega}{2\pi}$ represents the first Chern class of the determinant bundle of the spin$^c$ manifold $M$ in deRham cohomology, and $M$ has a Riemannian metric with mean convex boundary with $\scal> 2|\Omega|$, then $K\cw(M)<\infty$.
In this case, $\omega = \hat{A}\wedge \exp\big(\frac{\Omega}{4\pi}\big)$.
\end{example}


\end{document}